\documentclass{article}
\usepackage{amsmath,amssymb,amsfonts,theorem}
\usepackage{graphicx,epsfig,color} %
\usepackage{xspace,algorithmic}

\newcommand{\until}[1]{\{1,\dots, #1\}}

\newcommand{\subscr}[2]{#1_{\textup{#2}}}

\newcommand{\setdef}[2]{\{#1 \, : \; #2\}}

\newcommand{\map}[3]{#1: #2 \rightarrow #3}

\newcommand{\intersect}{\ensuremath{\operatorname{\cap}}}

\newcommand{\ba}{\begin{array}}
\newcommand{\ea}{\end{array}}

\newcommand{\be}{\begin{equation}}
\newcommand{\ee}{\end{equation}}

\newcommand{\mc}{\mathcal}

\newcommand{\Z}{\mathbb{Z}}

\newcommand{\1}{\mathbbm{1}}

\newcommand{\R}{\mathbb{R}}
\newcommand{\N}{\mathbb{N}}

\def\1{\mathbf{1}}

\def\Z{\mathbb{Z}}
\def\N{\mathbb{N}}
\def\R{\mathbb{R}}
\renewcommand{\natural}{\mathbb{N}}
\newcommand{\real}{\mathbb{R}}

\newcommand{\integernonnegative}{\mathbb{Z}_{\ge0}}

\newcommand{\Exp}{\mathbb{E}}
\renewcommand{\Pr}{\mathbb{P}}

\renewcommand{\deg}{d}

\newcommand{\V}{\mathcal{V}} 
\newcommand{\E}{\mathcal{E}}
\newcommand{\G}{\mathcal{G}}

\def\neigh{\mathcal{N}}

\def\L{\mathcal{L}}

\newcommand{\card}[1]{|#1|}
\newcommand{\xave}{\subscr{x}{ave}}
\newcommand{\trace}{\operatorname{tr}}
\newcommand{\diag}{\operatorname{diag}}
\newcommand{\circulant}{\operatorname{circ}}
\newcommand{\cayley}{\operatorname{cayl}}
\newcommand{\esr}{\operatorname{esr}}
\newcommand{\sr}{\operatorname{sr}}

\newcommand{\Reach}{\mathcal{R}}

\newcommand{\act}{\operatorname{Act}}
\newcommand{\rec}{\operatorname{Rec}}
\newcommand{\source}{\sigma}

\newtheorem{definition}{Definition}[section]
\newtheorem{theorem}{Theorem}[section]
\newtheorem{corollary}[theorem]{Corollary}
\newtheorem{lemma}[theorem]{Lemma}
\newtheorem{proposition}[theorem]{Proposition}
{ \theorembodyfont{\normalfont} 
\newtheorem{remark}{Remark}[section]
\newtheorem{example}[theorem]{Example}

}

\def\QEDopen{{\setlength{\fboxsep}{0pt}\setlength{\fboxrule}{0.2pt}\fbox{\rule[0pt]{0pt}{1.3ex}\rule[0pt]{1.3ex}{0pt}}}}
\def\QED{\QEDopen}

\title{Broadcast gossip averaging algorithms: interference and asymptotical error in large networks}
\author{Paolo Frasca \and Fabio Fagnani
\thanks{Paolo Frasca and Fabio Fagnani are with the Dipartimento
di Matematica, Politecnico di Torino, Corso Duca degli Abruzzi 24, 10129 Torino, Italy.
{\tt\small \{fabio.fagnani,paolo.frasca\}@polito.it}}%
}
\date{\today}

\begin{document}
\maketitle

\begin{abstract}
In this paper we study two related iterative randomized algorithms for distributed computation of averages. The first one is the recently proposed Broadcast Gossip Algorithm, in which at each iteration one randomly selected node broadcasts its own state to its neighbors. The second algorithm is a novel de-synchronized version of the previous one, in which at each iteration every node is allowed to broadcast, with a given probability: hence this algorithm is affected by interference among messages.
Both algorithms are proved to converge, and their performance is evaluated in terms of rate of convergence and asymptotical error: focusing on the behavior for large networks, we highlight the role of topology and design parameters on the performance. Namely, we show that on fully-connected graphs the rate is bounded away from one, whereas the asymptotical error is bounded away from zero. On the contrary, on a wide class of locally-connected graphs, the rate goes to one and the asymptotical error goes to zero, as the size of the network grows larger.
\end{abstract}


\section{Introduction}
When it comes to perform control and monitoring tasks through
networked systems, a crucial role has to be played by algorithms
for distributed estimation, that is algorithms to collectively
compute aggregate information from locally available data. Among
these problems, a prototypical one is the distributed computation
of averages, also known as the average consensus problem. In the
average consensus problem each node of a network is given a real
number, and the goal is for the nodes to iteratively converge to a
good estimate of the average of these initial values, by
repeatedly communicating and updating their states.

Recently, an increasing interest has been devoted among the
control and signal processing communities to {\em randomized} algorithms able to solve the
average consensus problem. This is motivated because randomized
algorithms may offer better performance or robustness with respect
to their deterministic counterparts. As well, randomized
algorithms may require less or no synchronization among the nodes,
a property which is often difficult to guarantee in the applications. Moreover,
it may very well happen that the communication network itself be
random, thus implying the need for a stochastic analysis. These
facts are especially true when communication is obtained through a
wireless network. For these reasons, the present paper will study
the performance of a pair of notable randomized algorithms, in
terms of their ability to approach average consensus.
Among randomized algorithms, researchers have devised {\em gossip}
algorithms, in which at each iteration only a random subset of the
nodes performs communication and update. Among these algorithms,
the Broadcast Gossip Algorithm has been recently proposed: at each
time step one node, randomly selected from a uniform distribution
over the nodes, broadcasts its current value to its neighbors. Each of its
neighbors, in turn, updates its value to a convex combination of its previous value and the received one. This algorithm may
seem to require a significant synchronization, since the choice of
the broadcasting node has to be done at the global level. However,
it has been observed that this communication model is equivalent,
up to a suitable scaling of time, to assume that each node
broadcasts at time instants selected by a private Poisson process.
Nevertheless, this equivalence is no longer true if broadcasting
takes a finite duration of time. When this happens a node is, with
non-zero probability, the target of more than one simultaneous
communication and destructive collision may occur. This is
especially true in wireless communications, which have to share
their communication medium. Hence, the practical applicability of
this algorithm in a distributed system resides either on the
possibility to incorporate some nontrivial collision detection
scheme, or on the assumption that communications are
instantaneous. The first contribution of this paper is to relax
this assumption to allow a communication model in which more than
one node can broadcast at the same time, possibly implying the
interference of attempted communications. Thus, we introduce a
novel distributed randomized algorithm for average computation,
facing the issue of interference in communication.

As a second contribution, we study both the original broadcast
algorithm and the novel one, in terms of their asymptotical
estimation error. Note, indeed, that the iterations of broadcast
algorithms do not preserve the average of states, and in general
do not converge to the initial average. Hence, it is crucial for
the application to estimate such bias. In this paper, we prove
that on sparse graphs with bounded degree, both algorithms are
asymptotically unbiased, in the sense that the asymptotical errors
go to zero as the network grows larger. Instead, on complete
graphs, both algorithms are asymptotically biased. Moreover, for
both algorithms we investigate significant trade-offs between
speed of convergence and asymptotical performance. Our results are
obtained via a mean square analysis, under the assumption that the
communication graph possesses some structural symmetries, namely,
that it is the Cayley graph of an Abelian group.

\subsection*{Related works}
In latest years, many papers have dealt with distributed
estimation and synchronization in networks. The latter problem,
which partly motivates the interest for our algorithms, has been
studied in several papers, also using approaches based on
consensus, as in~\cite{QL-DR:06,RC-AC-LS-SZ:08b}. Namely, an
increasing interest has been devoted to randomized averaging algorithms. An influential {\em pairwise}
gossip communication model is introduced in~\cite{DK-AD-JG:03}
and~\cite{SB-AG-BP-DS:06}. Moreover, the interest for wireless
networks has induced several authors to consider gossip models
based on {\em broadcast}, rather than on pairwise
communication~\cite{FF-SZ:08a,TCA-MEY-ADS-AS:09}. Other gossip
models have been developed, for instance,
in~\cite{FB-AGD-PT-MV:07,AGD-ADS-MJW:08,DM-DS:08,PD-FB-PT-MV:08}.

Broadcast and wireless communication inherently imply the issue of
interference among simultaneous communications. This has been
considered since early times: actually, our communication model
can be easily related with the \emph{slotted ALOHA} protocol
illustrated in~\cite{NA:85}. More recently, the negative effects
of interference on the connectivity of wireless networks have been
discussed in various papers, for
instance~\cite{OD-FB-PT:05,OD-MF-NM-RM-PT:06}.
The role of interference and message collisions in consensus
problems, and the effectiveness of countermeasures, has been
already investigated in the computer science
community~\cite{GC-MD-SG-NL-CN-TN:08}: in the latter paper,
consensus has to be achieved on a variable belonging to a finite
set. Several related papers about real-valued average consensus
problems have also appeared: we shall briefly review some of them.
A few paper are concerned with design of communication protocols
dealing with message collisions: for instance, \cite{SK-AS-RJT:07}
presents a data driven architecture which grants channel access to
nodes based on their local data values. In~\cite{YWH-AS:06}, a
consensus algorithm allowing simultaneous quantized transmissions
has been proposed: the issue of collisions is avoided by a
suitable data-dependent coding scheme. In~\cite{BN-AGD-MG:08},
local additive interference is considered, and under some
technical assumptions a collaborative consensus algorithm is
proposed, which allows simultaneous transmissions and ensures
energy savings. Besides interference, related robustness results
against data losses in distributed systems have been presented,
for instance, in~\cite{VS-MA-OS:06,FF-SZ:08}.
Finally, the paper~\cite{TCA-ADS-AGD:09} proposes a related
communication model, in which the broadcasted values are received
or not with a probability which depends on the transmitter and
receiver nodes, rather than on the activity of their neighbors.

\smallskip
Our results build on the general mean square analysis developed
for randomized consensus algorithms
in~\cite{FF-SZ:07ita,FF-SZ:08a}, and on results in algebra, linear
algebra and probability theory. Indeed, our main results will
assume that the network topology exhibits specific symmetries, in
the sense that it is represented by the Cayley graph of an Abelian
group. Cayley graphs have a long history in abstract
mathematics~\cite{LB:79}, and they have been recently used in
control theoretical applications, for instance
in~\cite{BR-RDA:04,FG-SZ:09}, to describe
communication networks. Assuming Abelian Cayley topologies is
motivated by their algebraic structure, which allows a formal
mathematical treatment, as well as by the potential applications.
Indeed, Abelian Cayley graphs are a simplified and idealized
version of communications scenarios of practical interest. In
particular, they capture the effects on performance of the strong
constraint that, for many networks of interest, communication is
local, not only in the sense of a little number of neighbors, but
also with a bound on the geometric distance among connected
agents. This is especially true for wireless networks: indeed,
Abelian Cayley graphs have been related, for instance
in~\cite{SB-AG-BP-DS:06,SR:07,RC-FG-SZ:09}, to other models for
wireless networks, as random geometric graphs or
disk-graphs~\cite{MF-RM:07}. Moreover, it has to be noted that
several topologies which appear in the applications are themselves
Abelian Cayley: for instance, complete graphs, rings, toroidal
grids, hypercubes. Finally, our main findings about the
asymptotical error are based on a result in~\cite{FF-JCD:10} about
the limit of the invariant vectors of sequences of stochastic
matrices: related results dealing with comparison of Markov chains
have appeared, for instance, in~\cite{JAF:91,PD-LSC:93}.

\subsection*{Paper structure}\label{sec:structure}
After presenting the averaging problem and the algorithms under
consideration in Section~\ref{sec:statement}, we develop our
analysis tools in Section~\ref{sec:prelim}. Later,
Sections~\ref{sec:NoColl} and~\ref{sec:Coll} are devoted to
analyze the original broadcast gossip algorithm and the novel one,
respectively. Some concluding remarks are presented in
Section~\ref{sec:conclusion}.

\subsection*{Notations and preliminaries}\label{sec:notations}
Given a set $\V$ of finite cardinality $\card{\V}=N$, we define a
graph on this set as $\G=(\V,\E)$, where $\E\subseteq\V\times\V$
(we exclude the presence of self-loops, namely edges of type
$(u,u)$). Given $u,v\in\V$, if $(v,u)\in\E$, we shall say that $v$
is an in-neighbor of $u$, and conversely $u$ is an out-neighbor of
$v$. We will denote by $\neigh^+_u$ and $\neigh^-_u$, the set of,
respectively, the out-neighbors and the in-neighbors of $u$. Also,
$\deg^+_u=\card{\neigh^+_{u}}$ and $\deg^-_u=\card{\neigh^-_{u}}$, are said to be the out-degree and the in-degree of node
$u$, respectively.
A graph whose nodes all have in-degree $k$ is said to be $k$-regular.
A graph is said to be (strongly) connected if for any pair of
nodes $(u,v)$, one can find a path, that is an ordered list of
edges, from $u$ to $v$. A graph is said to be symmetric, if $(u,v)\in\E$ implies $(v,u)\in \E$. In a symmetric
graph, being the neighborhood relation symmetrical, there is no
distinction between in- and out-neighbors and we will drop,
consequently, the index $+$ and $-$. We let $\1$ be the $N-$vector
whose entries are all $1$, $I$ be the $N\times N$ identity matrix,
and $\Omega:=I-N^{-1}\1\1^*$. Given a $N$-vector $a$, we denote by
$\diag(a)$ the diagonal matrix whose diagonal is equal to $a$. The
adjacency matrix of the graph $\G$, denoted by $A_{\G}$, is the
matrix in $\{0,1\}^{\V\times\V}$ such that ${A_\G}_{uv}=1$ if and
only if $(v,u)\in \E$. We also define the out-degree matrix as $D^+_{\G}:=\diag{(A_{\G}^*\1)}$, the in-degree matrix as
$D^-_{\G}:=\diag{(A_{\G}\1)}$, and the Laplacian matrix as
$L_\G=D^-_\G-A_\G$. The subscript $\G$ will be usually skipped for
the ease of notation. Given a matrix $M\in\R^{\V\times\V}$, we
define the graph ${\cal G}_M=(\V, \E_M)$ by putting
$(v,w)\in \E_M$ iff $v\neq w$ and $M_{wv}\neq 0$. A matrix $M$ is
said to be {\em adapted} to the graph $\G=(\V,\E)$ if $\G_M\subseteq \G$, that is if $\E_M\subseteq \E$.

\medskip
When it comes to compare two sequences $\{a_n\}_{n\in\natural}$
and $\{b_n\}_{n\in\natural}$, we shall write that $a_n=o(b_n)$ if
$\limsup_{n}{\frac{|a_n|}{|b_n|}}=0$, that $a_n=O(b_n)$ if
$\limsup_{n}{\frac{|a_n|}{|b_n|}}<+\infty$, and that
$a_n=\Theta(b_n)$ if there exist $\bar n\in\natural$ and positive
scalars $c_1,c_2$ such that $c_1 b_n\le a_n\le c_2 b_n,$ for $n\ge
\bar n.$

\medskip
Given a linear operator $\L$ from a vector space to itself, for
instance represented by a square matrix, we denote by $\sr(\L)$
its spectral radius, that is the modulus of its largest in
magnitude eigenvalue. Whenever $\sr(\L)=1,$ we shall define as
$\esr(\L)$ the modulus of the second largest eigenvalue in
magnitude.
\bigskip

\section{Broadcast gossip averaging algorithms}\label{sec:statement}

In this section we present the averaging problem and the
algorithms we are dealing with. Let us be given a graph
$\G=(\V,\E)$ and a vector of real values $x\in\real^\V$, assigned
to the nodes.
 Then, the averaging problem consists in approximating the average $\frac{1}{N}\sum_{v\in\V}x_v$, with the constraint that at each time step each node $v$ can
communicate its current state to its out-neighbors only. The simplest
solution to this problem consists in an iterative algorithm such
that $x(0)=x$, and for all $t\in\integernonnegative$, $x(t+1)=P
x(t)$, where $P\in \real^{\V\times\V}$ is a doubly stochastic
matrix adapted to $\G$. Provided the diagonal of $P$ is non-zero,
and the graph $\G$ is connected, the algorithm solves the
averaging problem, in the sense that for every $u\in \V$,
$$\lim_{t\to+\infty}x_u(t)=\xave(t),$$ where by definition $\xave(t)=\frac{1}{N}\sum_{v\in\V}x_v(t)$.

However, it is clear that this algorithm potentially requires
synchronous communication along all the edges of the graph. As
this requirement may be difficult to meet in real applications, in
this paper we shall study algorithms which require little or no
synchronization among the agents.
We assume from now on that the communication network be
represented by a graph, denoted by $\G=(\V,\E)$, whose
adjacency and Laplacian matrix will be denoted by $A$ and $L$,
respectively.

\bigskip
We start recalling the Broadcast Gossip Algorithm \cite{FF-SZ:08a,TCA-MEY-ADS-AS:09}. In this algorithm, at each time step one
node, randomly selected from a uniform distribution over the
nodes, broadcasts its current value to its neighbors. Its
neighbors, in turn, update their values to a convex combination of
their previous values and the received ones. More formally, we can
write the algorithm as follows. Note that the only design
parameter is the weight given to the received value in the convex
update.


\renewcommand\footnoterule{\hrule width \textwidth height .4pt}

\bigskip
\footnoterule
\vspace{.75\smallskipamount}
\noindent\hfill\textbf{Broadcast Gossip Algorithm}
\textendash{} Parameters: $q\in (0,1)$
\hfill\vspace{.75\smallskipamount}
\footnoterule\vspace{.75\smallskipamount}
\noindent For all $t\in\integernonnegative$,
\begin{algorithmic}[1]
\STATE Sample a node $v$ from a uniform distribution over $\V$%
\FOR {$u\in \V$}
    \IF{$\,u\in \neigh^+_v$}
    \STATE $x_u(t+1)=(1-q)x_u(t)+q x_v(t)$ \ELSE %
    \STATE $x_u(t+1)=x_u(t)$
    \ENDIF
\ENDFOR
\end{algorithmic}
\vspace{0.5\smallskipamount}\footnoterule\bigskip

This algorithm can also be written in the form of iterated matrix
multiplication. Let $v$ be the broadcasting node which has been
sampled at time $t$. Then, $x(t+1)=P(t)x(t),$ where
\begin{equation}\label{eq:P-BGA}
P(t)=I+q \sum_{u\in\neigh^+_v}(e_ue_v^*-e_ue_u^*),
\end{equation}
and $e_i$ is the $i$-th element of the canonical basis of
$\real^\V$. Clearly, at each time $t$, the matrix $P(t)$ is the
realization of a uniformly distributed random variable, depending
on the stochastic choice of the broadcasting node.

The Broadcast Gossip algorithm has received a considerable, and
has been extensively studied in~\cite{TCA-MEY-ADS-AS:09}: in that
paper, under the assumption that the communication graph is
symmetric, the algorithm is shown to converge, and its speed of
convergence is estimated. Instead, in this paper we shall
concentrate on another crucial analysis question. Since the
algorithm does not preserve the average of states $\xave$ through
iterations, how far from the initial average the convergence value
will be?

\bigskip
As we noted in the introduction, the practical interest of the BGA
algorithm depends on the assumption that the transmissions are
instantaneous, and reliable. In an effort towards more realistic
communication models, we propose a modification of the Broadcast
Gossip Algorithm, which has the feature of dealing with the issue
of finite-length transmissions, and consequent packet losses due
to collisions.

At each time step, \emph{each} node wakes up, independently with
probability $p$, and broadcasts its current state to all its
out-neighbors. It is clear that some agents can be the target of more
than one message: in this case, we assume that a destructive
collision occurs, and no message is actually received by these
agents. Moreover, interference prevents the broadcasting nodes
from hearing any others ({\it half-duplex} constraint\footnote{The
{half-duplex} constraint is assumed throughout the paper: however,
dropping it would imply minimal changes in the analysis.}). If an
agent $u\in\V$ is able to receive a message from agent $v$, it
updates its state to a convex combination with the received value,
similarly to the standard BGA.

More formally, the algorithm is as follows.
%

\medskip\footnoterule\vspace{.75\smallskipamount}
\noindent\hfill\textbf{Collision Broadcast Gossip Algorithm}
\textendash{} Parameters: $q\in (0,1)$, $p\in (0,1)$
\hfill\vspace{.75\smallskipamount}
\footnoterule\vspace{.75\smallskipamount}
\noindent For all $t\in\integernonnegative$,
\begin{algorithmic}[1]
\STATE let $\act$ 
be the random set defined by: %
  for every $v\in\V$, $\Pr[v\in \act]=p$%
\STATE let $\rec:=\setdef{u\in\V}{\card{\neigh^-_u \intersect \act}=1, u\not \in \act}$%
\STATE for all $u\in \rec$, let $\source(u)$ be the only $v\in \V$ such that $v\in \act\intersect \neigh^-_u$%
\FOR {$u\in \V$}
    \IF{$\,u\in \rec$}
    \STATE $x_u(t+1)=(1-q)x_u(t)+q x_{\source(u)}(t)$ \ELSE %
    \STATE $x_u(t+1)=x_u(t)$
    \ENDIF%
\ENDFOR
\end{algorithmic}
\vspace{0.5\smallskipamount}\footnoterule\bigskip

Also the latter algorithm can be written as matrix multiplication, defining
\begin{equation}\label{eq:P-CBGA}
P(t)=I+q \sum_{(v,u)\in\,\act\times\rec}(e_ue_v^*-e_ue_u^*).
\end{equation}

Both algorithms can actually be rewritten in the following
graph-theoretic way. Let ${\cal G}(t)$ be the subgraph of $\cal G$
depicting the communications taking place at a certain instant
$t$: the pair $(u,v)$ is an edge in ${\cal G}(t)$ iff $v$ successfully
receives a message from $u$ at time $t$. Denote by $A(t)$, $D(t)$, $L(t)$ the
adjacency, degree and Laplacian matrices, respectively, of $\G(t)$. Clearly, for both algorithms:
\begin{equation}\label{eq:P-general}
P(t)=I-qL(t)
\end{equation}

%
Several questions are natural for the collision-prone CBGA
algorithm, in comparison with its synchronous collision-less
counterpart. Does the algorithm converge? How fast? Does it
preserve the average of states? If not, how far it goes? Is
performance poorer because of interferences?

We are going to answer the analysis questions we have posed, via a
mean square analysis of the algorithm. Our interest will be mostly
devoted to the properties of algorithms for large networks. To
this goal, we shall often assume to have a sequence of graphs
$\G_N$ of increasing order $N\in\N$, and we shall consider, for
each $N\in\natural$, the corresponding matrix $P(t)$, which
depends on $\G$ and then on $N$. Thus we will focus on studying
the asymptotical properties of the algorithms as $N$ goes to
infinity.

\section{Mathematical models and techniques}\label{sec:prelim}
In this section we lay down some mathematical tools that can be
used to analyze gossip and other randomized algorithms. Namely, in
Subsection~\ref{sec:MSA} we review the mean square analysis
in~\cite{FF-SZ:08a}, which is going to be applied to the BGA and
CBGA algorithms in Sections~\ref{sec:NoColl} and~\ref{sec:Coll},
respectively. Later, in Subsection~\ref{sec:Cayley-intro} we
introduce Abelian Cayley graphs and their properties, and in
Subsection~\ref{sec:LocalPerturb} we present perturbation results
about sequences of stochastic matrices and their invariant
vectors.

\subsection{Mean square analysis}\label{sec:MSA}
Motivated by the interpretation of the broadcast algorithms as
iterated multiplications by random matrices, given in
Equations~\eqref{eq:P-BGA} and~\eqref{eq:P-CBGA}, in this
subsection we shall recall from~\cite{FF-SZ:08a} some definitions
and results for the analysis of randomized schemes, in which the
vector of states $x(t) \in \R^\V$ evolves in time following an
iterate $x(t+1)=P(t) x(t),$ where
$\{P(t)\}_{t\in\integernonnegative}$ is a sequence of i.i.d.
stochastic-matrix-valued random variables. Consequently, $x(t)$ is
a stochastic process.

The sequence $P(t)$ is said to achieve {\em probabilistic
consensus} if for any $x(0)\in \R^\V$, it exists a scalar random
variable $\alpha$ such that almost surely $\lim_{t\to
\infty}x(t)=\alpha\1.$ The following result, proved in
\cite{FF-SZ:08a}, is a simple and effective tool to prove the
convergence of these randomized linear algorithms. Let
$\bar{P}:=\Exp[P(t)]$.
\begin{proposition}[Probabilistic consensus criterion]\label{prop:ConvCrit}
If $P(t)$ is such that the graph $\G_{\bar{P}}$ is strongly
connected and, for all $v\in\V$, almost surely $P(t)_{vv}>0$, then
$P(t)$ achieves probabilistic consensus.
\end{proposition}
For the rest of this section, we shall assume that the assumptions
of Proposition~\ref{prop:ConvCrit} are satisfied. Then, in order
to describe the speed of convergence of the algorithm, let us
define $d(t):=N^{-1}\|x(t)-\xave(t)\1\|^2,$ and the rate of
convergence as
\begin{equation}\label{eq:R}
    R:=\sup_{x(0)}\limsup_{t\rightarrow +\infty} \Exp[d(t)]^{1/t}.
\end{equation}

Let us consider the (linear) operator $\map{\L}{\real^{\V\times \V}}{\real^{\V\times \V}}$ such that $$\L(M)=\Exp[P(t)^* M P(t)].$$
Notice that $\Exp[d(t)]=\Exp[x^*(t)\Omega x(t)]$, and that $\Exp[x^*(t)\Omega x(t)]=x^*(0)\Delta(t) x(0),$ with $\Delta(t)=\L^t(\Omega).$
Let $\Reach$ be the reachable space of the pair $(\L,\Omega)$. Then, $$R=\sr(\L_{|\Reach}),$$
and it has been proved in~\cite{FF-SZ:08a} that the rate can be estimated in terms of eigenvalues of $N\times N$ matrices, as
\begin{equation}\label{eq:bounds}
\esr(\overline{P})^2\leq R \leq \sr(\L(\Omega)).
\end{equation}

It is clear that, if all $P(t)$ matrices are doubly stochastic,
$x(t)$ converges to the initial average of states $\xave(0)$. If
they are not, as in the cases we are studying in this paper, it is
worth asking how far is the convergence value from the initial
average. To study this bias in the estimation of the average, we
let $\beta(t)=|\xave(t)-\xave(0)|^2$, and we define a matrix $B$
such that
\begin{equation}\label{eq:B-def}
\lim_{t\to \infty}\Exp[\beta(t)]=x(0)^*B x(0).\end{equation}

Let $Q(t)=P(t-1) \ldots P(0)$, so that $x(t)=Q(t)x(0).$ The
convergence of the algorithm is equivalent to the existence of a
random variable $\rho$, taking values in $\R^\V$, such that
$\lim_{t\to \infty} Q(t)=\1 \rho^*$. This implies that
$$B=\Exp[\rho\rho^*]-2 N^{-1}\Exp[\rho]\1^*+N^{-2}\1\1^*,$$ where
$\Exp[\rho]$ and $\Exp[\rho\rho^*]$ are the eigenvectors relative
to 1 of $\bar P$ and $\L$, respectively. In particular, if $\bar
P$ is doubly stochastic, then
\begin{equation}\label{eq:B-simpleCase}
B=\Exp[\rho\rho^*]-N^{-2}\1\1^*,\end{equation}

For any $N$-dimensional matrix $\Delta$,
$\Exp[\rho\rho^*]=\frac{1}{\1^*\Delta\1}\lim_{t\to \infty} \L^t(\Delta):$
as a consequence, in the latter case the matrix $B$ can be computed as
\begin{equation}\label{eq:B-compute}
B=N^{-2}\lim_{t\to\infty}{\L^t(\1\1^*)}-N^{-2}\1\1^*.\end{equation}
Instead of computing $B$, it may be easier and significant to
obtain results about some functional of $B$, for instance the
spectral norm $\|B\|_2$ or the trace $\trace(B)$. The latter
figure is of interest because,
if we assume that the initial values $x_i(0)$ are i.i.d. random variables with zero mean and variance $\sigma^2$, then
\begin{equation}\label{eq:trB-meaning}
\Exp[x(0)^*B x(0)]=\sigma^2 \trace{(B)}.\end{equation}

Motivated by our interest for the properties of the algorithms on large networks, and by Equation~\eqref{eq:trB-meaning}, we state the following definition.
\begin{definition}
Given a sequence of graphs $\G_N$, a randomized algorithm $P(t)$ is said to be {\em asymptotically unbiased} if $\displaystyle \lim_{N\to\infty}\trace(B).$
\end{definition}

\subsection{Abelian Cayley graphs}\label{sec:Cayley-intro}
A special family of graphs is that of Abelian Cayley graphs, which
are graphs representing a group, as follows. Let $G$ be an Abelian
group, considered with the additive notation, and let $S$ be a
subset of $G$. Then, the {\em Abelian Cayley graph} generated by
$S$ in $G$ is the graph $\G(G,S)$ having $G$ as node set and
$\E=\setdef{(g,h)\in G\times G}{h-g\in S}$ as edge set.
Note that the graph $\G(G,S)$ is symmetric if and only if $S$ is
inverse-closed, and is connected if and only if $S$ generates the
group $G$.
As well, a notion of Abelian Cayley {\em matrix} can be defined. Given a
group $G$ and a generating vector $\pi$ of length $\card{G}$, we
shall define the Cayley matrix generated by $\pi$ as
$\cayley(\pi)_{hg}=\pi_{h-g}.$ Correspondingly, for a given Cayley
matrix $M$, we shall denote by $\pi^{M}$ the generating vector of
the Cayley matrix $M$. Clearly, the adjacency matrices of $G$-Cayley graphs are $G$-Cayley matrices.
Abelian Cayley graphs and matrices enjoy important properties: we refer the reader to~\cite{LB:79,AT:99} for more details.

%
%
%

\begin{example}
Abelian Cayley graphs encompass several important examples.
\begin{enumerate}
\item\label{ex:Complete} The {\em complete} graph on $N$ nodes,
that is the graph where each node is directly connected with every
other node, is $\G(\Z_N,\Z_N\setminus \{0\})$;
\item\label{ex:Ring} The {\em circulant} graphs (resp. matrices)
are Abelian Cayley graphs (resp. matrices) on the group $\Z_N$; we
shall denote the circulant matrix generated by $\pi$ as
$\circulant(\pi)$. For instance, he {\em ring} graph is the
circulant graph $\G(\Z_N,\{-1,1\})$; its adjacency matrix is
$A=\circulant([0,1,0,\ldots,0,1])$ and its Laplacian is
$L=\circulant([2,-1,0,\ldots,0,-1])$. For a ring, the eigenvalues
of $L$ are
$\{2(1-\cos{\left(\frac{2\pi}{N}l\right)})\}_{l\in\Z_N}$, and in
particular
$\lambda_1=\frac{4\pi^2}{N^2}+o\left(\frac1{N^3}\right)$ as
$N\to+\infty.$ \item The square {\em grids} on a $d$-dimensional
torus are $\G(\Z_n^d,\{e_i,-e_i\}_{i\in\until{d}}),$ where $e_i$
are elements of the canonical basis of $\real^d$. In particular,
the {\em $n$-dimensional hypercube} graph is
$\G(\Z_2^n,\{e_i\}_{i\in \until{n}}).$
\end{enumerate}
\end{example}

Notice how all examples above are naturally forming a sequence of
graphs indicized by the number of nodes $N$. Special cases for
which we will be able to prove asymptotically unbiasedness in the
following, is when the generating set $S$ is finite and ``kept fixed'' as in the ring graph. Precisely we consider the following
general example

\begin{example}\label{exCayley} Start from an infinite lattice $\V={\Z}^d$ and fix a finite $S\subseteq \Z^d\setminus\{0\}$
generating $\Z^d$ as a group. For every integer $n$, let $V_n=[-n,
n]^d$ considered as the Abelian group $\Z_{2n+1}^d$ and let $\mc
G^{(n)}$ be the Cayley Abelian graph generated by $S_n=S\cap [-n,
n]^d$. Notice that all graphs $\mc G^{(n)}$ have the same
generating set $S$ for $n$ sufficiently large, in particular they
have the same degree. Moreover, by the assumption made on $S$, all
of them are strongly connected. Rings and grids fit in this
framework.
\end{example}

\bigskip
%
A $G$-Cayley structure for the communication graph $\G$ has deep
consequences on the mean square analysis of randomized consensus
algorithms. In
particular, it is easy to see that if $C$ is a $G$-Cayley matrix,
then also $\L(C)$ is $G$-Cayley.
To exploit this property, let us define a sequence of matrices by
the following recursion. Let $\Delta(0)=N^{-1}\1\1^*$ and
$\Delta(t+1)=\L(\Delta(t))$. Since $\1\1^*$ is a Cayley matrix on
any Abelian group, the above fact implies that for every
$t\in\integernonnegative$, the matrices $\Delta(t)$ are
$G$-Cayley. Thus, the sequence $\Delta(t)$ can be equivalently
seen as the sequence of the corresponding generating vectors
$\pi(t)=e_0^*\Delta(t)$. We shall refer to this vector as the {\em
MSA vector}. Since $\L$ is linear, the MSA vector evolution can be
written as a matrix multiplication $\pi(t+1)=M \pi(t)$. Clearly,
\begin{equation}\label{eq:R-compute-Cayley}
R=\esr(M).
\end{equation}
Moreover, $M$ is $\ast$-stochastic, and if we let $\pi'$ be the right invariant vector of $M$, that is
$$\begin{cases}\pi'=M\pi'\\ \1^*\pi'=\1^*, \end{cases}$$
then, $B$ can be computed using the fact that stochastic Cayley matrices are also doubly stochastic and Equation~\eqref{eq:B-compute}, obtaining
\begin{equation}\label{eq:B-compute-Cayley}
B=\frac{1}{N}\cayley{\left(\pi'-\frac{1}{N} \1\right)},\quad
\trace B=\pi'_0-\frac1N.
\end{equation}
The following simple result will be useful later on.
\begin{lemma}\label{lemma: Lgraph} Let $\G$ be $G$-Cayley, and suppose that $P(t)_{uu}\geq \delta>0$ almost surely.
Then,
$$\G_{A}\subseteq \G_{M^*}\subseteq \G_{A+A^*+A^*A}$$
\end{lemma}\begin{proof}
A straightforward computation shows that
$$M_{uv}=\sum_k\Exp[P(t)_{k+v,u}P(t)_{k0}]$$
Hence, $M_{uv}>0$ implies that there exists $k$ such that
$P(t)_{k+v,u}>0$ and $P(t)_{k0}$. If $k=0$, this yields
$A_{vu}>0$. If $k+v=u$, then $A_{u-v,0}>0$ or also $A_{uv}>0$.
Finally, if both cases above do not happen, then, $A_{k+v,u}>0$
and $A_{k+v,v}>0$ which yields $(A^*A)_{uv}>0$. The second inclusion
is thus proven. To prove the first one, notice that, by the
assumption made, $M_{uv}\geq \delta\, \Exp[P(t)_{vu}]$. This
completes the proof. \end{proof}

Notice that in the BGA and CBGA examples we can always apply Lemma
\ref{lemma: Lgraph} with $\delta =1-q$.



\subsection{Local perturbation of stochastic matrices}\label{sec:LocalPerturb}
In this section we recall a perturbation result presented
in~\cite{FF-JCD:10} and which will be used later to estimate the
trace of the matrix $B$ and to prove asymptotic unbiasedness for
sequences of Cayley graphs.

We assume we have fixed an infinite universe set ${\cal V}$, an
increasing sequence $V_n$ of finite cardinality subsets of ${\cal
V}$ such that $\cup_nV_n={\cal V}$ and a sequence of irreducible
stochastic matrices $P^{(n)}$ on the state spaces $V_n$ with the
following stabilizing property: for every $i\in\mc V$, there exist
$n(i)\in \N$ such that $i\in V_{n(i)}$ and \be
P^{(n)}_{ij}=P^{(n(i))}_{ij}\,,\quad\forall n\geq n(i)\,,\;\forall
j\in V_{n(i)}\ee This property allows us to define, in a natural
way, a limit stochastic matrix on ${\cal V}$. For every $i,j\in
\mc V$, we define
\be\label{infinityP}P^{(\infty)}_{ij}=\left\{\ba{ll}
P^{(n(i))}_{ij}\quad &{\rm if}\; j\in V_{n(i)}\\ 0\quad &{\rm
otherwise}\ea\right.\ee The sequence of stochastic matrices $P^{(n)}$ is
said to be {\em weakly democratic} if the corresponding invariant vectors $\pi^{(n)}$ are such that, for all $i\in\V$,
$\pi^{(n)}_i\to 0$ for $n\to +\infty$. Fix now a finite subset
$W\subseteq\cap_nV_n$  and another sequence of irreducible
stochastic matrices $\tilde P^{(n)}$ on $V_n$ such that
\be\ba{ll}\tilde
P^{(n)}_{ij}=P^{(n)}_{ij}\quad\forall i\in V_n\setminus W\,,\;\forall j\in V_n\\[8pt]
\tilde P^{(n)}_{ij}=\tilde P^{(1)}_{ij}\quad \forall\, i\in W\,,\;
\forall j\in V_1 \ea\ee In other terms, $\tilde P^{(n)}$ can be
seen as a perturbed version of $P^{(n)}$ with the perturbation
confined to the fixed subset $W$ and stable (it does not change as
$n$ increases). Also for this perturbed sequence we can define,
following (\ref{infinityP}), the asymptotic chain $\tilde
P^{(\infty)}$. The following result has been proven in~\cite{FF-JCD:10}.
\begin{theorem}\label{theoweakdemocracy}
Suppose that $\tilde P^{(\infty)}$ and $\tilde P^{(\infty)}$ are
both irreducible. Then, if $P^{(n)}$ is weakly democratic, also
$\tilde{P}^{(n)}$ is weakly democratic.
\end{theorem}

In the sequel we will use this result to prove the convergence to
$0$ of the trace in (\ref{eq:B-compute-Cayley}).

\section{Broadcast without collisions}\label{sec:NoColl}
In this section we present a comprehensive analysis of the
Broadcasting Gossip Algorithm, in terms of both rate of
convergence and bias.
The following result characterizes the convergence properties of the algorithm, extending~\cite[Lemma~2 and~4]{TCA-MEY-ADS-AS:09} to directed networks.

\begin{proposition}[Convergence of BGA algorithm]\label{prop:RateNoColl}
Consider BGA algorithm. Let $\G$ be any connected graph, $L$ its Laplacian matrix, and $\lambda_1$ the smallest positive eigenvalue of $L$. Then
\begin{align}\label{eq:barPNoColl}
{\bar P}=I-q N^{-1}L
\end{align}
\begin{align}\label{eq:UNoColl}
\L(\Omega)=\Omega-q (1-q) N^{-1} (L+L^*)+q N^{-2} (L^* \1\1^*+\1\1^*L) -q^2 N^{-2} (D^+-A)(D^+-A^*).
\end{align}
In particular, BGA algorithm achieves probabilistic consensus.
\end{proposition}
\begin{proof}
First, we compute $\bar P$, as
$$
\Exp[P(t)]=\frac1N \sum_{v\in\V} \big(I+q \sum_{u\in\neigh^+_v} (e_ue_v^*-e_ue^*_u) \big)=I-q N^{-1}L.
$$
Then, by Proposition~\ref{prop:ConvCrit}, the algorithm achieves probabilistic consensus. To compute $\L(\Omega)=\Exp[P(t)^*\Omega P(t)]$, we notice that $e^*_ie_j=\delta_{ij}$, and we compute
\begin{align*}
\Exp[P(t)^*P(t)]=&\frac1N \sum_{v\in\V} \left(I+q \sum_{u\in\neigh^+_v} (e_ue_v^*-e_ue^*_u) \right)^* \left(I+q \sum_{w\in\neigh^+_v} (e_we_v^*-e_we^*_w) \right)\\ =& I - \frac1N q\, (L+L^*) + q^2 \frac1N \sum_{v\in\V} \sum_{u\in\neigh^+_v}\left( e_ve_v^* -e_ue_v^* -e_ve_u^*+e_ue_u^*\right)\\
=& I - \frac1N q\, (L+L^*) + q^2 \frac1N (L+L^*);
\end{align*}
and
\begin{align*}
\Exp[P(t)^*\1\1^*P(t)]=&\frac1N \sum_{v\in\V} \left(I+q \sum_{u\in\neigh^+_v} (e_ue_v^*-e_ue^*_u) \right)^* \1\1^*\left(I+q \sum_{w\in\neigh^+_v} (e_we_v^*-e_we^*_w) \right)\\
=& \11^*- \frac1N q\, (L^*\1\1^*+\1\1^*L) + q^2 \frac1N \sum_{v\in\V} \sum_{u,w\in\neigh^+_v} \left(e_ve_v^* -e_ve_w^* -e_ue_v^*+e_ue_w^*\right)\\
=& \11^*- \frac1N q\, (L^*\1\1^*+\1\1^*L)\\&\quad + q^2 \frac1N
\left( \sum_{v\in\V} (d^+_v)^2 e_ve_v^* -
\sum_{v\in\V} d^+_v \sum_{w\in\neigh^+_v} e_ve_w^*
- \sum_{v\in\V} d^+_v \sum_{u\in\neigh^+_v} e_ue_v^*+ \sum_{v\in\V} \sum_{u,w\in\neigh^+_v} e_ue_w^*\right)\\
=& \1\1^* - \frac1N q\, (L^*\1\1^*+\1\1^*L) + q^2 \frac1N ((D^+)^2-D^+A^*-AD^++AA^*).
\end{align*}
\end{proof}

The next corollary provides explicit bounds on the convergence rate, assuming the communication graph to be symmetric. 
\begin{corollary}\label{cor:RateNoColl-undirected}
Under the assumptions of Proposition~\ref{prop:RateNoColl}, if $\G$ is symmetric, then
\begin{align}\label{eq:UndirectedNoColl}
{\bar P}=I-q N^{-1}L \qquad
\L(\Omega)=\Omega-2 q (1-q) N^{-1}L-q^2 N^{-2} L^2.
\end{align}
In particular, the convergence rate can be estimated as
\begin{align}\label{eq:Rundirected}
1-\frac{2 q}{N}\lambda_1\le R \le 1-\frac{2 q (1-q)}{N}\lambda_1.
\end{align}
\end{corollary}
\begin{proof}
Specializing Proposition~\ref{prop:RateNoColl}, and combining~\eqref{eq:UndirectedNoColl} with~\eqref{eq:bounds}, we get
\begin{align*}
&R\ge \left(1-\frac{q}{N}\lambda_1\right)^2\ge 1-\frac{2 q}{N}\lambda_1\\
&R\leq 1-\frac{2 q (1-q)}{N}\lambda_1-\frac{q^2}{N^2}\lambda_1^2\le 1-\frac{2 q (1-q)}{N}\lambda_1.
\end{align*}
\end{proof}

Once we have understood that the algorithm converges, and how
fast, it is worth to discuss to which value it converges. It is
clear that the average is not preserved through each iteration,
although the average can sometimes be preserved in expectation. Indeed, $\Exp\left[\xave(t+1)| x(t)\right]= N^{-1}\1^*\bar P x(t)$, so that if $\G$ is symmetric,
\begin{align*}
\Exp\left[\xave(t)\right]=\xave(0), \qquad \forall t>0.
\end{align*}
However, this weak preservation property does not imply that the expected estimation error
$\Exp[\beta(t)]$ be zero, neither in finite time nor as time goes to infinity.
In facts, the following example, derived from~\cite{FF-SZ:07ita},
shows that $\lim_t\Exp[\beta(t)]$ can be positive and,
furthermore, it can be bounded away from zero, uniformly in the
network size $N$.
\begin{example}[Complete graph]\label{ex:ComplNoColl}\normalfont
Let $\G$ be a complete graph, and consider the BGA algorithm.
In this case, the operator $\L$ can be computed explicitly, giving
\begin{align*}
&R=(1-q)^2\\
&B=\frac{q}{2-q}\frac{1}{N}(I-\frac{\1\1^*}{N}).
\end{align*}
Namely, $\trace{B}=\frac{q}{2-q}\left(1-\frac{1}{N}\right)$ and
hence the {\em BGA is not asymptotically unbiased on the complete
graph}. Note that by changing the parameter $q$, one can trade off
speed and estimation bias: this is numerically investigated in
Figure~\ref{fig:complete-nocoll-q-BR}. Similar trade-offs will be
considered throughout the paper.
\end{example}
\begin{figure}[htb]
\center
\includegraphics[width=7cm]{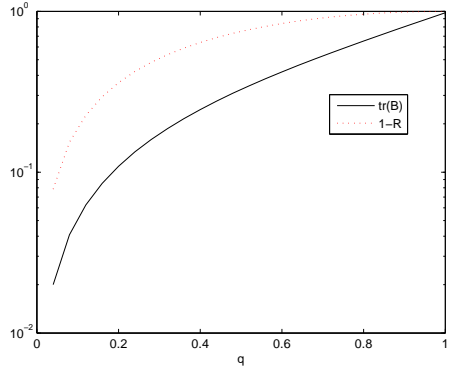}
\caption{Trade-off in the BGA between $\trace(B)$ and $R$ as functions of $q$, for a complete graph with $N=30$.}\label{fig:complete-nocoll-q-BR}
\end{figure}

\bigskip
Next, we focus on {\em Abelian Cayley graphs}.
\begin{lemma}\label{lemma:C+T-CayleyNoColl}
Consider the BGA algorithm and let the communication graph $\G$ be
Abelian Cayley with degree $d$. Then, the MSA vector $\pi$ evolves
as
\begin{align}\label{eq:EvolPiNoColl-Matrix}
\pi(t+1)=(C+T) \pi(t),
\end{align}
where the matrix $$C=I-2 \frac{q}{N} (L+L^*)$$ is Cayley and $T$
is a matrix such that $\displaystyle T=\frac{q^2}{N}\tilde T$ where $\tilde T$
does not depend neither on $q$ nor explicitly on $N$ but only on
$S$, and is such that the number of non-zero rows is
at most $d^2$ and non-zero columns is at most $d^2-d+1$.
\end{lemma}
\begin{proof}
Using the fact that all matrices are Abelian Cayley, we obtain
\begin{align*}
\Delta(t+1)=&\Exp[P(t)^*\Delta(t)P(t)]\\
=&\Delta(t)-q\Exp[L(t)^*\Delta(t)]-q\Exp[\Delta(t)L(t)]+q^2\Exp[L(t)^*\Delta(t)L(t)]
\\
=&\left(I-\frac{q}{N} (L+L^*)\right)\Delta(t)+q^2\frac1N
\sum_{g\in G} \sum_{h: h-g\in S}\sum_{k: k-g\in S}
(e_ge_h^*-e_he_h^*)\Delta(t) (e_ke_g^*-e_ke_k^*)
\end{align*}

Since $\pi(t)=\Delta(t)e^0$ (being $\Delta(t)$ symmetric for every
$t$), we easily obtain from above that
\begin{align*}\pi(t+1)=&(I-  q\frac1N (L+L^*) )\pi(t)+\frac1N q^2
\sum_{g\in G} \sum_{h: h-g\in S} \sum_{k: k-g\in S}\pi_{h-k}(t)
(e_ge_g^*-e_ge_k^*-e_he_g^*+e_he_k^*)e_0\\
=&(I-  q\frac1N (L+L^* )\pi(t)+\frac1N q^2[\sum_{h: h\in S}
\sum_{k: k\in S}\pi_{h-k}(t)(e_0-e_h) +\sum_{g\in -S} \sum_{h:
h-g\in S} \pi_{h}(t)(e_h-e_g)] \,.\end{align*} From this we
immediately see that the non-zero elements of $\tilde T$ have row
indices in $(S-S)\cup S$ and column indices in $S-S$. Hence the
result follows.

\end{proof}
Note that, in general, $C$ will not be a stochastic matrix since
it may be negative on the diagonal, however for large enough $N$
(and $d$ fixed) it is surely a stochastic matrix. Note, moreover,
that the entries of the matrix $T$
in~\eqref{eq:EvolPiNoColl-Matrix} are proportional to $N^{-1}$
and, moreover, the number of the non-zero entries is upper bounded
by $(1+d^2)^2$, which does not depend on $N$. Hence, we expect
that, as $N$ diverges, $T$ would become negligible, and the MSA
would depend on the matrix $C$ only. This would imply the
unbiasedness of the algorithm, because it is immediate to remark
that the invariant vector of $C$ is $N^{-1}\1.$ This fact can be
actually stated as the following result.
\begin{theorem}[Unbiasedness of BGA]\label{th:BGA-unbiased}
Fix a finite $S\subseteq \Z^d\setminus\{0\}$ generating $\Z^d$ as
a group. For every integer $n$, let $V_n=[-n, n]^d$ considered as
the Abelian group $\Z_{2n+1}^d$ and let $\mc G^{(n)}$ be the
Cayley Abelian graph generated by $S_n=S\cap [-n, n]^d$. On the
sequence of $\mc G^{(n)}$ the BGA is asymptotically unbiased.
\end{theorem}
\begin{proof}
The idea is to apply the perturbation result Theorem~\ref{theoweakdemocracy} to the sequence of matrices $C^*$ and
$(C+T)^*$. Notice that ${\G}_{C^*}=\mc G^{(n)}$ while ${\G}_{(C+T)^*}\supseteq\mc G^{(n)}$ by Lemma~\ref{lemma: Lgraph}.
Hence, ${\cal G}_{C^*}$ and ${\G}_{(C+T)^*}$ are both strongly
connected. This also implies that the limit graph on $Z_n^d$ of
the two sequences ${\G}_{C^*}$ and ${\G}_{(C+T)^*}$ both
contain $\G^{(\infty)}$ which is simply the Abelian Cayley
graph on $Z_n^d$ generated by $S$ which is strongly connected by
the assumption made. Finally, notice that $C^*$ is Abelian Cayley,
hence obviously weakly democratic, while Lemma
\ref{lemma:C+T-CayleyNoColl} guarantees that $(C+T)^*$ is a finite
perturbation of $C^*$ in the sense of Section
\ref{sec:LocalPerturb}. Hence also $(C+T)^*$ is weakly democratic.
This yields, by~\eqref{eq:B-compute-Cayley},
$$\trace B=|N^{-1}-\pi'_0|\le N^{-1}+\pi'_0\to 0.$$
\end{proof}

\bigskip

The results about Cayley graphs can be made more specific in the following example.
\begin{example}[Ring graph]\label{ex:RingNoColl}\normalfont
Note that in this case the Cayley graph is circulant.
For the ring graph, Proposition~\ref{prop:RateNoColl} implies that, for $N$ large enough,
$$1-q \frac{8\pi^2}{N^3}\le R\le 1-q(1-q) \frac{8\pi^2}{N^3},$$
and namely ${R=1-\Theta(\frac{1}{N^3})}$.
Specializing the proof of Lemma~\ref{lemma:C+T-CayleyNoColl}, the evolution of $\pi(t)$ can be written as
\begin{align*}
\pi_j(t+1)=&\left(1-\frac{4q}{N}+\frac{q^2}{N}\pi^{A^2}_{j}\right)\pi_j(t)+2 \left(\frac{q}{N}-\frac{q^2}{N} \pi^A_{j}\right)(\pi_{j-1}(t)+\pi_{j+1}(t))+
\frac{q^2}{N}\left(2 \pi_0(t)+\pi_2(t)+\pi_{-2}(t)\right)\delta_{0j},
\end{align*}
that is $\pi(t+1)=(C+T)\pi(t)$, with
$$C=\circulant{(1-\frac{4q}{N}, 2\frac{q}{N}, 0, \ldots, 0, 2\frac{q}{N})}$$
and
$$T=\frac{q^2}{N}\left(
                      \begin{array}{cccccccc}
                        4 & 0 & 1 & 0 & \dots & 0 & 1 & 0 \\
                        -2 & 0 & -2 & 0 & \ldots &  & \ldots & 0 \\
                        0 & 0 & 1 & 0 & \ldots &  & \ldots & 0 \\
                        0 & \ldots &  &  &  &  & \ldots & 0\\
                        0 & \ldots &  &  &  &  & \ldots & 0\\
                        0 & \ldots &  &  &  &  & \ldots & 0\\
                        0 & \ldots &  &  & \ldots & 0 & 1 & 0 \\
                        -2 & 0 &  \ldots &  & \ldots & 0 & -2 & 0 \\
                      \end{array}
                    \right).$$
Thanks to these explicit formulas, we can numerically compute the
rate and bias. Namely, about the rate we obtain that
$\esr(C)<\esr(C+T)$, and $\esr(C+T)=1-\Theta(N^{-3})$. This means
that the perturbation $T$ does not significantly affect the rate
for large $N$. Moreover, since $\esr(C)=1-\Theta(N^{-3})$ and
$\esr(C+T)-\esr(C)=1-\Theta(N^{-4})$ we argue that actually
$$R=1-q\frac{8\pi^2}{N^3}+O\left(\frac1{N^{4}}\right) \qquad \text{as $t\to\infty$}.$$
On the other hand, about the bias we obtain that 
$\trace{(B)}=\Theta(N^{-1})$.
We can also numerically evaluate $B$ and $R$ as functions of $q$:
these results are shown in Figure~\ref{fig:ring-nocoll-BR}. Note
that choosing $q$ we can trade off asymptotical error and
convergence rate. \hfill\QED
\end{example}
\begin{figure}[htb]
\center\includegraphics[width=7cm]{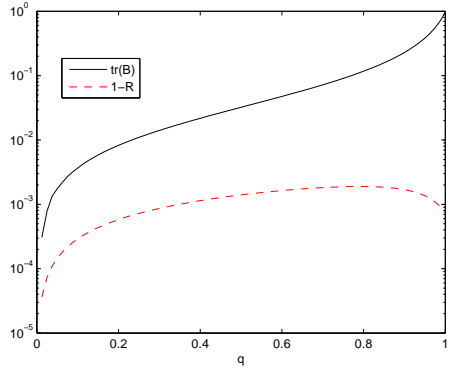}
\caption{Trade-off in the BGA between $\trace{(B)}$ and $R$ for a ring graph with $N=30$. }\label{fig:ring-nocoll-BR}
\end{figure}

\begin{remark}[Pareto optimality]
Combining Examples~\ref{ex:ComplNoColl} and~\ref{ex:RingNoColl},
we obtain that $\trace B$ is proportional to
$\frac{1-\sqrt{R}}{1+\sqrt{R}}\big(1-\frac1N\big)$ on a complete
graph, and to $N^3(1-R)$ on a ring graph, considering the
approximation for large $N$. In Figure~\ref{fig:pareto} we plot
$\trace B$ against $R$, for both graphs, up to a suitable scaling
with respect to $N$. The fact that both curves are monotone
decreasing highlights the fact that every value of $q$ is Pareto
optimal, in the sense that one can not improve one of the two
objectives without making the other worse.
\end{remark}
\begin{figure}[htb]
\center\includegraphics[width=7cm]{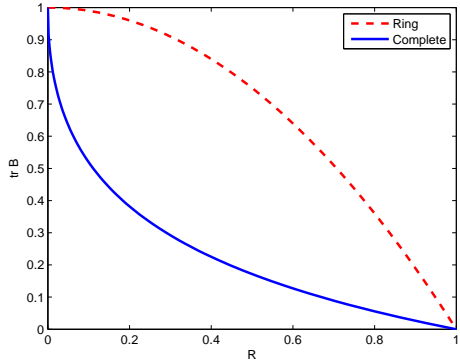}
\caption{$\trace B$ as a function of $R$ on complete and ring graphs.}\label{fig:pareto}
\end{figure}

\begin{example}[Random geometric graph]\label{ex:rggNoColl}
In order to take into account the locality constraint on
connectivity in real-world networks, several models of random
geometric graphs have been proposed, as accounted
in~\cite{MF-RM:07}. In this example we consider sequences of
random geometric graphs, based on the following construction. For
all $N\in\N$, we sample $N$ points $\{z_i\}_{i\in\Z_N}$ from a
uniform distribution over a unit square $[0,1]^2$, and we draw an
edge $(i,j)$ between nodes $i,j\in\Z_N$ when $\|z_i-z_j\|\le 0.8
\sqrt{\frac{\log{N}}{N}}.$ On these realizations we run the BGA
algorithm until convergence is reached, up to a small tolerance
threshold, and in this way we compute an approximation of
$\lim_{t\to\infty}\beta(t).$ The results are plotted in
Figure~\ref{fig:compare-bias-NoColl}, in comparison with the
analogous quantity on the complete and ring graph. It appears from
simulations that, as $N$ diverges, $\beta$ is $\Theta(1)$ on the
complete graph, whereas it is $\Theta(N^{-1})$ on the ring graph,
and $\Theta(N^{-1/2})$ on the random geometric graph. These
evidences are in accordance with the theoretical results, and
suggest their extension to other families of geometric graphs.
\end{example}
\begin{figure}[htb]
\center\includegraphics[width=9cm]{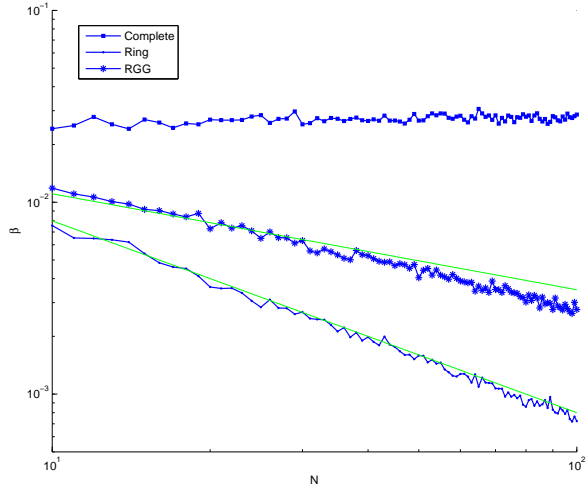}
\caption{The plot shows the asymptotical bias $\beta$ as a
function of $N$, computed by simulations on sequences of complete,
random geometric, and ring graphs. Plotted values are the
average over 1000 runs. See text for more
information.}\label{fig:compare-bias-NoColl}
\end{figure}

\section{Broadcast with collisions}\label{sec:Coll}
In this section we present a comprehensive analysis of the
Collision Broadcast Gossip Algorithm, in terms of both rate of
convergence and bias. After proving a general convergence result,
we focus on complete graphs in Subsection~\ref{sec:complete-coll},
on Abelian Cayley graphs in Subsection~\ref{sec:Cayley-Coll}, and
finally on ring graphs in Subsection~\ref{sec:ring-coll}. Our main
finding is that the performance of the CBGA, in terms of both
speed and bias, is close to the BGA: in this sense we may claim
the robustness of broadcast gossip algorithms to local
interferences.

\begin{proposition}[Convergence of CBGA algorithm]\label{prop:Coll}
Consider the CBGA algorithm. Let $\G$ be any connected graph, and $L$ its Laplacian matrix. Then,
\begin{equation}\label{eq:barPgeneral}
\bar{P}=I-q p (1-p)^{D^-}L,
\end{equation}
where $((1-p)^{D^-})_{ij}=(1-p)^{D^-_{ij}}$ for every $i$ and $j$. In
particular, the CBGA algorithm achieves probabilistic consensus.
\end{proposition}
\begin{proof}
The probability of having at time $t$ a successful transmission
from $v$ to $u$ is $\Pr[A_{uv}(t)=1]=p(1-p)^{\deg^-_u} A_{uv}$.
Hence,
\begin{equation}\label{eq:success}
\Exp(A(t))=p(1-p)^{D^-}A\,.
\end{equation}
Now,
\begin{align*}
\bar P=\Exp[P(t)]&= \Exp[I+q(A(t)-D(t))]=I+qp(1-p)^{D^-}[A-{D^-}].
\end{align*}
Note that  $\G_{\bar P}$ is strongly connected: by
Proposition~\ref{prop:ConvCrit}, we can conclude the convergence
of the CBGA.
\end{proof}

\medskip
\noindent{\bf Remark:} Formula~(\ref{eq:barPgeneral}) is simpler
if the graph is $d$-regular, because in that case
\begin{equation}\label{eq:barPregular}
    \bar{P}=I-q p (1-p)^d L.
\end{equation}
Denote by that $\lambda_1$ is the smallest nonzero eigenvalue of $L$. Then $\esr({\bar P})=1-qp(1-p)^d \lambda_1$,
and this together with~\eqref{eq:bounds} and~$[1-q p(1-p)^d
\lambda_1]^2\ge 1-2 q p(1-p)^d \lambda_1$ leads to
\begin{equation}\label{eq:lowerbound-R}
R \ge 1-2 q p(1-p)^d \lambda_1
\end{equation}
as a lower bound for the rate of convergence.
As a function of $p$, this lower bound is minimal whenever $\Pr[A_{ij}(t)=1]$ is maximal, that is for
$p$ equal to $p^*=\frac{1}{d+1}.$

Hence, natural questions are: is this bound tight? is $p^*$ the best choice to improve the convergence rate?
The content of the next section will answer positively these questions for complete and ring graphs.

\subsection{Complete graphs}\label{sec:complete-coll}
A thorough analysis can be carried on for the complete graph.
\begin{proposition}[Complete graph]\label{prop:CompleteColl}
Let $x(t)$ evolve following the CBGA algorithm, and let $R$ and $B$ be defined by~\eqref{eq:R} and~\eqref{eq:B-simpleCase}, respectively. Then,
\begin{align*}
&R=1-q(2-q)N p (1-p)^{N-1}\\
&B=\frac{q}{2-q}\frac{1}{N}(I-\frac{\1\1^*}{N}).
\end{align*}
Namely, $\trace(B)=\frac{q}{2-q}\left(1-\frac{1}{N}\right)$ and then the {\em CBGA is not asymptotically unbiased on the complete graph}.
\end{proposition}
\begin{proof}
It is immediate that in the complete graph either one node
communicates to every others, or no node communicates.  In the
latter case, $P(t)=I$. The former case, which has probability
$Np(1-p)^{N-1}$, leads to $P(t)=P_b$, where $P_b=I+q\sum_{u\neq
v}(e_ue_v^*-e_ue_u^*)$ and $v$ is the realization of a random
variable  uniformly distributed over the nodes. Consequently, the
analysis is quite close to the one of
Example~\ref{ex:ComplNoColl}. We note that
$$\Exp[P_b]=I+q\frac{\1\1^*-NI}{N}$$ and that
\begin{align*}
\Exp[P(t)^*P(t)]=&N p(1-p)^{N-1}\Exp[P_b^*P_b]+(1-N p(1-p)^{N-1})I\\
=& N p(1-p)^{N-1}[I+2q \frac{\1\1^*-N I}{N} +q^2\left(\frac{N-1}{N}I-\frac{2}{N}(\1\1^*-I)+\frac{N-1}{N}I\right)]\\
&+(1-N p(1-p)^{N-1})I\\
=& Np(1-p)^{N-1}[I-2q (1-q)(I-\frac{\1\1^*}{N})]+(1-N p(1-p)^{N-1})I\\
=& [1-2q (1-q)N p(1-p)^{N-1}]I+2q (1-q)N p(1-p)^{N-1}\frac{\1\1^*}{N}.
\end{align*}
and
\begin{align*}
\Exp[P(t)^*\1\1^*P(t)]=&
N p(1-p)^{N-1}\big[\1\1^*+\frac{q^2}{N}\big(
(N-1)^2 I-2 (N-1)(\1\1^*-I)\\&
+(N-2)(\1\1^*-I)+(N-1)I\big)\big]+(1-Np(1-p)^{N-1})\1\1^*\\
=& q^2 N^2p (1-p)^{N-1}I+(1-q^2 Np
(1-p)^{N-1})\1\1^*.
\end{align*}
This means that the application of $\L$ keeps invariant the subspaces generated by $I$ and $\1\1^*$, and the linear operator $\L$ can be represented by the matrix
\begin{equation*}
    \left(%
\begin{array}{cc}
  1-2q(1-q)Np (1-p)^{N-1} & q^2 Np (1-p)^{N-1} \\
  2q(1-q)Np (1-p)^{N-1} & 1-q^2 Np (1-p)^{N-1} \\
\end{array}%
\right).
\end{equation*}
The eigenvalues of this matrix are $1$ and
\begin{align*}
R&=1-q^2 Np (1-p)^{N-1}-2q(1-q)Np (1-p)^{N-1}\\
&= 1-q(2-q)N p (1-p)^{N-1},
\end{align*}
and the eigenspace relative to eigenvalue 1 is spanned by
\begin{equation*}q N p(1-p)^{N-1}
\left(
\begin{array}{c}
  q \\
  2(1-q) \\
\end{array}\right).
\end{equation*}
Since $E[\rho\rho^*]$ belongs to this eigenspace, and $\1^*E[\rho\rho^*]\1=1,$ we conclude that
\begin{align*}
B=E[\rho\rho^*]-N^{-2}\1\1^*=\frac{q}{2-q}\frac{1}{N}\big(I-\frac{\1\1^*}{N}\big).
\end{align*}
\end{proof}


Some remarks are in order about the role of the parameters $p,q$ in the CBGA algorithm on the complete graph.
\begin{remark}[Optimization]
The convergence rate $R$ as a function of $p$ is optimal for $p^*=1/N.$ The optimal value
$R(p^*)=1-q(2-q)(1-\frac{1}{N})^{N-1}$ is increasing in $N$ and tends to $1-q(2-q)\frac{1}{e}$
when $N$ goes to infinity. Instead, if we fix $p=\bar{p}$, $R(\bar{p})$ tends to 1 as $N\to \infty$. 
On the other hand, $B$ is the same as for the BGA in Example~\ref{ex:ComplNoColl} and is independent of $p$.

From the design point of view, it is clear that $p$ has to be
chosen equal to $N^{-1}$, optimizing the speed. Instead, choosing
$q$ we trade off speed and asymptotic displacement: this trade-off
is numerically investigated in
Figure~\ref{fig:complete-coll-q-BR}.
\end{remark}

\begin{figure}[ht]
\center
\includegraphics[width=7cm]{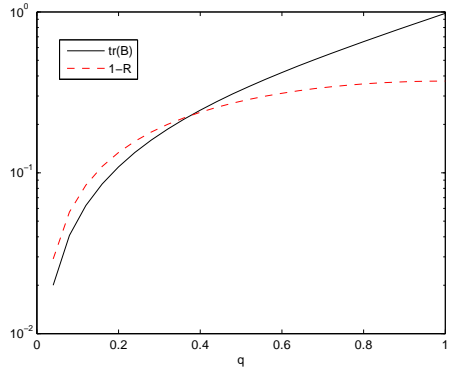}
      \caption{Trade-off in the CBGA between $\trace(B)$ and $R$ as functions of $q$, for a complete graph with $N=30$ and $p=1/N$.}\label{fig:complete-coll-q-BR}
\end{figure}

\subsection{Abelian Cayley graphs}\label{sec:Cayley-Coll}
In this subsection we consider the CBGA on Abelian Cayley graphs with bounded degree. The next result characterizes the mean square analysis of the algorithm.
\begin{lemma}\label{lemma:C+T-CayleyColl}
Consider the CBGA algorithm and let the communication graph $\G$
be Abelian Cayley with degree $\deg$. Then, the MSA vector $\pi$
evolves as
\begin{align}
\pi(t+1)=(C+T) \pi(t),
\end{align}
where the matrix $$C=(I-qp (1-p)^\deg L^*)(I-qp (1-p)^\deg L)
$$ is Cayley and $T$
is a matrix whose entries do not depend on $N$ and such that the
number of non-zero rows is at most $d^2(d+1)^2$ and non-zero
columns is at most $d^2$.
\end{lemma}

\begin{proof}
%
Using commutativity of Abelian Cayley matrices and
(\ref{eq:success}), we obtain that
\be\label{deltacoll} \Delta(t+1)=(I- q p(1-p)^d (L+L^*))\Delta(t)+
q^2\Exp[L(t)^*\Delta(t)L(t)] \ee Passing to the generating vector,
\be\label{deltacoll2} \pi(t+1)=(I- q p(1-p)^d (L+L^*))\pi(t)+
q^2f(\pi(t) \ee where
$$f(\pi(t)):=\Exp[L(t)^*\Delta(t)L(t)]e^0=\underbrace{\Exp[L(t)]^*\Delta(t)\Exp[L(t)]e^0}_{=:f_1(\pi(t))}+
\underbrace{\Exp[L(t)^*\Delta(t)L(t)]e^0-\Exp[L(t)]^*\Delta(t)\Exp[L(t)]e^0}_{=:f_2(\pi(t))}
$$

Now,
\begin{equation}\label{F1}
f_1(\pi(t))=p^2(1-p)^{2d}L^*L\pi(t)
\end{equation}
and
\begin{equation}\label{F2}
\begin{array}{rcl}[f_2(\pi(t))]_l&=&\sum_{hk}\left[\Exp[L(t)_{hl}L(t)_{k0}\right]-\Exp[L(t)_{hl}]\Exp[L(t)_{k0}]\pi(t)_{h-k}\\[10pt]
&=&\sum_{tk}\left[\Exp[L(t)_{k+t,l}L(t){k0}\right]-\Exp[L(t)_{k+t,l}]\Exp[L(t){k0}]\pi(t)_{t}
\end{array}
\end{equation}
Notice now that, by the way the model has been defined, we have
that $A_{ij}(t)$ and $A_{hk}(t)$ are independent whenever
$N^-(i)\cap N^-(h)=\emptyset$ or equivalently $i-h\not\in S-S$. In
this case, also $L_{ij}(t)$ and $L_{hk}(t)$ are independent.
Therefore, the double summation in (\ref{F2}), can be restricted
to $k\in S\cup\{0\}$ and $t\in S-S$. Consequently, the values of
$l$ for which $[f_2(\pi(t))]_l\neq 0$ can be restricted to
$(S\cup\{0\})+S-S-(S\cup\{0\})$. Plugging (\ref{F1}) inside
(\ref{deltacoll2}) and using the information on the structure of
$f_2(\pi(t))$ obtained above, the result follows.

\end{proof}

Thanks to Lemma~\ref{lemma:C+T-CayleyColl}, we can argue a result analogous to Theorem~\ref{th:BGA-unbiased}.

\begin{theorem}[Unbiasedness of CBGA]\label{th:CBGA-unbiased}
Fix a finite $S\subseteq \Z^d\setminus\{0\}$ generating $\Z^d$ as
a group. For every integer $n$, let $V_n=[-n, n]^d$ considered as
the Abelian group $\Z_{2n+1}^d$ and let $\mc G^{(n)}$ be the
Cayley Abelian graph generated by $S_n=S\cap [-n, n]^d$. On the
sequence of $\mc G^{(n)}$ the CBGA is asymptotically unbiased.
\end{theorem}
\begin{proof}
The idea is to apply the perturbation result Theorem
\ref{theoweakdemocracy}to the sequence of matrices $C^*$ and
$(C+T)^*$. Notice that ${\cal G}_{C^*}=\mc G_{A+A^*+A^*A}\supseteq
\mc G^{(n)}$ while ${\cal G}_{(C+T)^*}\supseteq\mc G^{(n)}$ by
Lemma \ref{lemma: Lgraph}. Hence, ${\cal G}_{C^*}$ and ${\cal
G}_{(C+T)^*}$ are both strongly connected. As in the proof of
Theorem~\ref{th:BGA-unbiased} this also implies that the limit
graph on $Z_n^d$ of the two sequences ${\cal G}_{C^*}$ and ${\cal
G}_{(C+T)^*}$ are also strongly connected. Finally, notice that
$C^*$ is Cayley Abelian, hence obviously weakly democratic, while
Lemma \ref{lemma:C+T-CayleyNoColl} guarantees that $(C+T)^*$ is a
finite perturbation of $C^*$ in the sense of Section
\ref{sec:LocalPerturb}. Hence also $(C+T)^*$ is weakly democratic.
This yields
$$\trace B=|N^{-1}-\pi'_0|\le N^{-1}+\pi'_0\to 0.$$
\end{proof}

The form of $T$ depends on the graph, and is complicated to compute in general. To give an example, in the next subsection we shall compute $T$ for the ring graph.

%
\subsection{Ring graphs}\label{sec:ring-coll}
In this subsection, we specialize the results about about Cayley
graphs to the case of ring graphs. In particular, we obtain tight
bounds on its convergence rate. The following lemma can be proved
on the lines of Lemma~\ref{lemma:C+T-CayleyColl} by a direct
computation, which we omit.
\begin{lemma}\label{lemma:C+T-RingColl}
Let $\G$ be a ring graph on $N\ge 9$ nodes. Let $C$ be the doubly stochastic symmetric circulant matrix
$C=\circulant([1-2k_1-2k_2, k_1, k_2, 0, \ldots, 0, k_2, k_1 ])$ with
\begin{align*}
&k_1=2qp(1-p)^2(1-2qp(1-p)^2)\\
&k_2=q^2 p^2(1-p)^4,
\end{align*}
and let
\begin{align*}
&T=q^2p(1-p)^2\cdot\\
&\left(%
\begin{array}{cccccccc}
  4-6 p(1-p)^2 & 4p(1-p)^2 & (1-p)^3 & 0 & \ldots  & (1-p)^3 & 4p(1-p)^2 \\
  4p(1-p)^2-2 & (1-6(1-p)^2)p & 2(2p-1)(1-p)^2  & 0 & \ldots & 0 & -p(1-p)^2 \\
  -p(1-p)^2 & 4p(1-p)^2-2p & (1-p)(1-6p(1-p)) & 0  & \ldots & 0 & 0 \\
  0 & (1-(1-p)^2)p & -2p(1-p)(2p-1) & 0 & \ldots & 0  & 0 \\
  0 & 0 & p^2(1-p) & 0 & \ldots & 0  & 0 \\
  0 & 0 & 0 & 0 & \ldots & 0  & 0 \\
  \ldots & \ldots & \ldots & \ldots & \ldots & \ldots  & \ldots \\
  0 & 0 & 0 & 0 & \ldots & p^2(1-p)  & 0 \\
  0 & 0 & 0 & 0 & \ldots & -2p(1-p)(2p-1) & (1-(1-p)^2)p  \\
  -p(1-p)^2 & 0 & 0 & 0  & \ldots &  (1-p)(1-6p(1-p)) & 4p(1-p)^2-2p \\
  4p(1-p)^2-2 & -p(1-p)^2 & 0 & 0  & \ldots & 2(1-p)^2 (2 p-1) &  (1-6(1-p)^2)p \\
\end{array}%
\right)
\end{align*}
%
Then, the MSA vector evolves as $ \pi(t+1)=(C+T)\pi(t).$
\end{lemma}

Numerical computations based on Lemma~\ref{lemma:C+T-RingColl} show that $\trace(B)=\Theta(N^{-1})$, that is the asymptotical error has the same
dependence on $N$ as for the BGA.



\medskip
Next result is about the rate of convergence, and shows that the bound on the convergence rate based on $\bar P$ is tight for the ring graph.
\begin{proposition}[Ring graph - Rate]\label{prop:RateColl-ring}
Given a ring graph and CBGA algorithm, we have
\begin{align*}
\bar{P}&=I-q p (1-p)^2 L;\\
\L(\Omega)&=\Omega-2 q(1-q)p (1-p)^2 L-q^2 p(1-p)^2 N^{-1} L^2+q^2 p(1-p)^2 N^{-1} p \circulant(\tau),
\end{align*}
where
$$\tau=[2(p-2), 6-4p+p^2,-3(2-2p+p^2),2-4p+3 p^2, 0, \ldots, 0, 2-4p+3 p^2, -3(2-2p+p^2), 6-4p+p^2].$$
%
Namely,
\begin{align*}
1- q  p (1-p)^2 \frac{8\pi}{N^2} \le R &\le 1- q (1-q) p (1-p)^2 \frac{8\pi}{N^2} \qquad \text{ for $N$ large enough}.
\end{align*}
\end{proposition}

\begin{proof}
Computing $\bar P$ is easy. In order to compute $\L(\Omega)$, we first compute, using Lemma~\ref{lemma:C+T-RingColl},
\begin{align*}
\L(I)=& \circulant((C+T) (1, 0, \ldots,0)^*)\\
=& \circulant(1- 2 k_1-2 k_2+ q^2 p (1-p)^2 (4-6 p (1-p)^2), k_1+q^2 p (1-p)^2 (4 p (1-p)^2-2), \\
 &\qquad k_2- q^2 p (1-p)^2  p (1-p)^2, 0, \ldots,0, k_2- q^2 p (1-p)^2  p (1-p)^2,  \\
&\qquad k_1+q^2 p (1-p)^2 (4 p (1-p)^2-2))\\
=&\circulant(1- 4 q (1-q) p (1-p)^2 , 2 p(1-p)^2 q (1-q),
0, \ldots,0, 2 p(1-p)^2 q (1-q))\\
=&-2 q(1-q)p (1-p)^2 L,
\end{align*}
and
\begin{align*}
\L(\1\1^*)=&\circulant((C+T) \1 )=\1\1^*+\circulant(T\1).
\end{align*}
Then, by straightforward computations,
\begin{align*}
\L(\Omega)=&\Omega-2 q(1-q)p (1-p)^2 L- \circulant(T\1)\\
=&\Omega-2 q(1-q)p (1-p)^2 L-q^2 p(1-p)^2 N^{-1} L^2 +q^2 p^2(1-p)^2 N^{-1} \circulant(\tau),
\end{align*}
provided we define $\tau$ as in the statement. 

%
Note that $\tau^*\1=0$ and that the eigenvalues of the circulant matrix $\circulant(\tau)$ are
$$\left\{\tau_0+ 2 \sum_{i=1}^{4} \tau_i \cos{\left( \frac{2 \pi}{N} i l \right)}\right\}_{l\in\Z_N},$$
and those of $L$ are
$\left\{2\left(1-\cos{\left( \frac{2 \pi}{N} l \right)}\right)\right\}_{l\in\Z_N}.$
As a consequence, the eigenvalues of
$\L(\Omega)$ are $\mu_0=0$, and
\begin{align*}
\mu_l=&1+ 2 p(1-p)^2 [-2 q(1-q) - 3q^2 N^{-1}+ \tau_0 q^2 p N^{-1} ]\\
&+2 p(1-p)^2 [ 2 q(1-q) +4 q^2 N^{-1} + \tau_1 q^2 p N^{-1}]\cos{\left( \frac{2 \pi}{N} l \right)}+2 q^2 p(1-p)^2 N^{-1}[\tau_2 p-1]\cos{\left( \frac{4 \pi}{N} l \right)}\\
&+ 2\tau_3 q^2 p^2(1-p)^2 N^{-1}\cos{\left( \frac{6 \pi}{N} l \right)}+ 2\tau_4 q^2 p^2(1-p)^2 N^{-1}\cos{\left( \frac{8 \pi}{N} l \right)},
\end{align*}
for $l=1,\ldots,N-1.$
Because of the $N^{-1}$ factors, we have that, as $N\to\infty$,
$$\mu_l=1-4 q(1-q) p(1-p)^2 \frac{2 \pi^2 l^2}{N^2}+ o\left(\frac{l^2}{N^2}\right).$$
Then, for $N$ large enough,
$$\esr \L(\Omega)=\max_{l\in \Z_N\setminus\{0\}} \mu_l=1-8 q(1-q) p(1-p)^2 \frac{\pi^2}{N^2}+ o\left(\frac{1}{N^2}\right).$$
Finally, the bound on the rate follows from~\eqref{eq:bounds}.
\end{proof}

\begin{remark}
The speed of convergence for the Collision Broadcasting Gossip is one order faster than the Broadcasting Gossip. This is not surprising, since in the former case the average number of activated nodes per round is $Np$, instead of $1$.
\end{remark}

\begin{remark}[Rate Optimization]
Remarkably, for $N$ large enough both the upper and the lower bound on the rate show the same dependence on $p$. Thus, they can be simultaneously optimized by taking $p^*=1/3.$
The behavior for large $N$ can also be investigated by numerical computations, showing that
$\esr(C)<\esr(C+T)$, and $\esr(C+T)=1-\Theta(N^{-2})$. This means that the perturbation $T$ does not significantly affect the rate for large $N$. Moreover, since $\esr(C)=1-\Theta(N^{-2})$ and $\esr(C+T)-\esr(C)=1-\Theta(N^{-3})$, we argue that actually
$$R=1-qp(1-p)^2\frac{8\pi^2}{N^2}+O\left(\frac1{N^{3}}\right).$$ This is reminiscent of a similar observation about the BGA algorithm in Remark~\ref{ex:RingNoColl}.

We can also numerically evaluate $B$ and $R$ as functions of $p$ and $q$: these results are shown in Figure~\ref{fig:ring-coll-BR}. Note that the asymptotical error does not depend significantly on $p$: this implies, from the design point of view, that $p$ can be chosen equal to $p^*=1/3$, optimizing the convergence rate. Instead, choosing $q$ we can trade off asymptotical error and convergence rate.
\end{remark}
\begin{figure}[ht]
\center
\includegraphics[width=7cm]{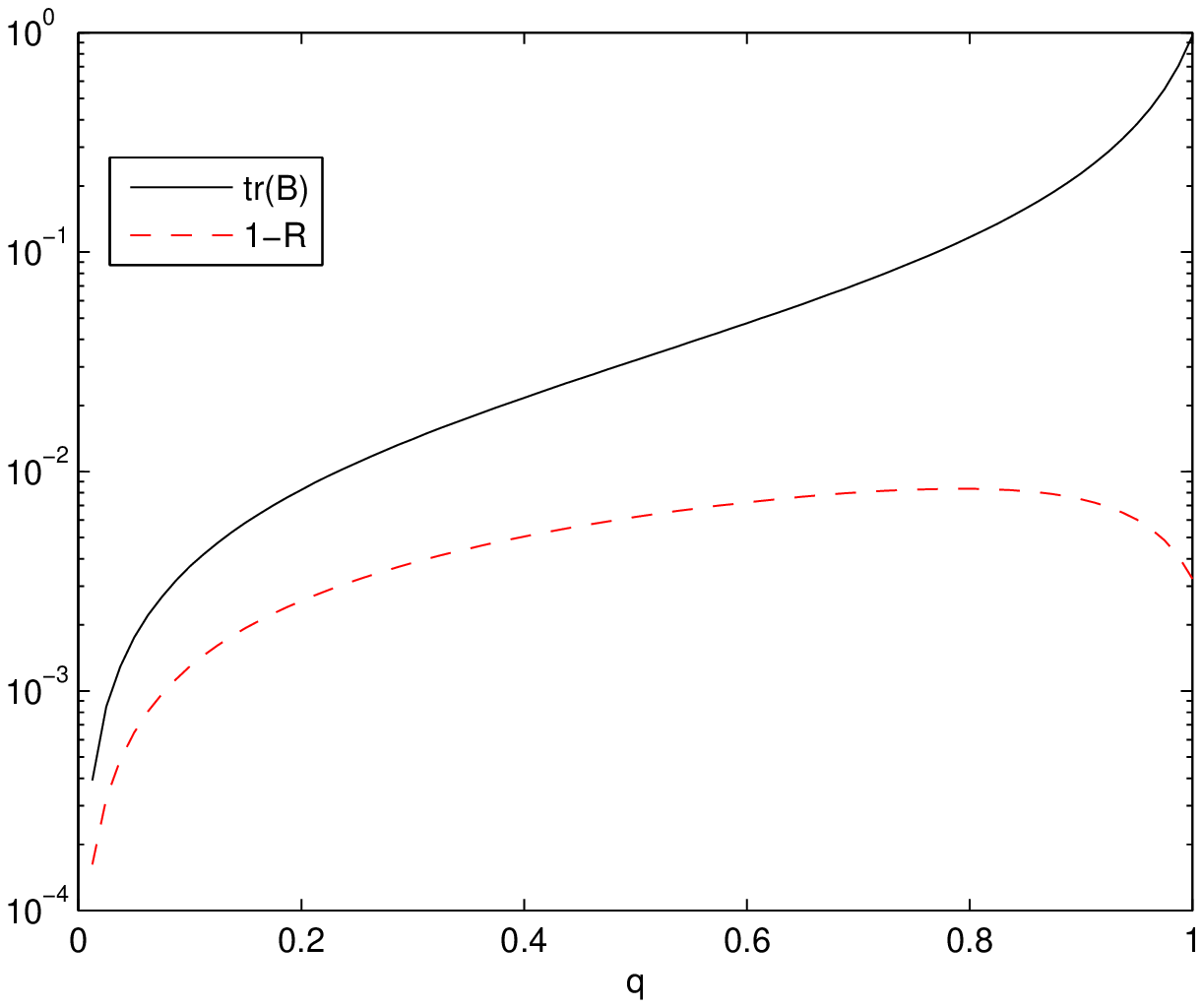}
\includegraphics[width=7cm]{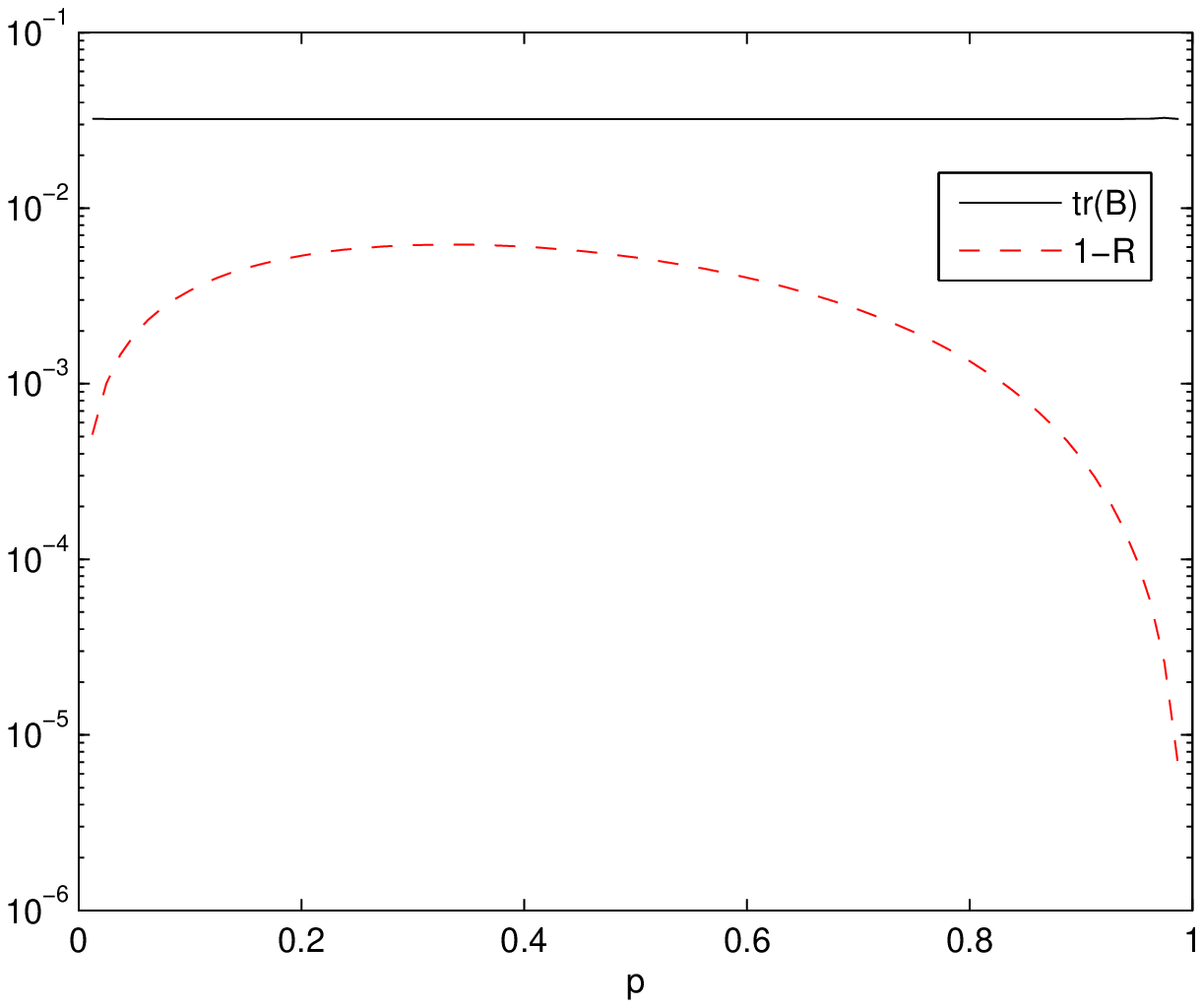}
\caption{Dependence in the CBGA of $\trace{(B)}$ and $R$ for the ring graph. Left plot assumes $N=30$, $p=1/3$, right plot assumes $N=30$, $q=1/2$.}\label{fig:ring-coll-BR}
\end{figure}

%

\section{Conclusion}\label{sec:conclusion}
This paper has been devoted to study gossip algorithm for the
estimation of averages, based on iterated broadcasting of current
estimates, focusing on their capability of providing an unbiased
estimation.
In this framework, we presented a novel broadcast gossip
algorithm, dealing with interference in communications, whose
impact is studied in the paper. Our results, obtained under
specific symmetry assumptions about the network topology, allow to
conjecture an interesting picture of the performance of broadcast
gossip algorithms on real world networks, in terms of achievable
precision and robustness to interference. Interferences have a
negative effect on the rate of convergence, which can be minimized
by a suitable choice of the broadcasting probability $p$. As known
for static consensus algorithms based on diffusion, the rate
degrades on large locally connected graphs. Instead, interferences
have minimal effect on the asymptotical error, which depends on
the network topology: on large highly connected graphs, the
algorithms provide biased estimations, whereas on locally
connected graphs the estimation bias goes to zero as the network
grows larger.

A better understanding of the role of the network topology in the
trade-off between speed and achievable precision might come from
the following extensions of our work. In this paper two
simultaneous technical assumptions have been made: the graphs are
Abelian Cayley, thus in particular vertex-transitive, and each
node broadcasts with the same probability. Future work should
consider non-vertex-transitive networks of nodes with non-uniform
broadcasting probabilities. Furthermore, it seems interesting to
study the performance of the algorithms on families of expander
graphs and small-world networks, which ensure a good scaling of
the rate of convergence, while keeping the degree of nodes
bounded.

\bigskip\noindent
{\bf Acknowledgements:} The authors wish to thank Sandro Zampieri
for many fruitful discussions on the issues studied in this paper.


\begin{thebibliography}{10}

\bibitem{NA:85}
N.~Abramson.
\newblock {Development of the ALOHANET}.
\newblock {\em IEEE Transactions on Information Theory}, 31(2):119--123, 1985.

\bibitem{TCA-ADS-AGD:09}
T.~C. Aysal, A.~D. Sarwate, and A.~G. Dimakis.
\newblock Reaching consensus in wireless networks with probabilistic broadcast.
\newblock In {\em Allerton Conf. on Communications, Control and Computing},
  Allerton, IL, USA, September 2009.

\bibitem{TCA-MEY-ADS-AS:09}
T.~C. Aysal, M.~E. Yildiz, A.~D. Sarwate, and A.~Scaglione.
\newblock Broadcast gossip algorithms for consensus.
\newblock {\em IEEE Transactions on Signal Processing}, 57(7):2748--2761, 2009.

\bibitem{LB:79}
L.~Babai.
\newblock Spectra of {C}ayley graphs.
\newblock {\em J. Combin. Theory Ser. B}, 27(2):180--189, 1979.

\bibitem{FB-AGD-PT-MV:07}
F.~Benezit, A.~G. Dimakis, P.~Thiran, and M.~Vetterli.
\newblock Gossip along the way: Order-optimal consensus through randomized path
  averaging.
\newblock In {\em Allerton Conf. on Communications, Control and Computing},
  Monticello, IL, US, Sep 2007.

\bibitem{SB-AG-BP-DS:06}
S.~Boyd, A.~Ghosh, B.~Prabhakar, and D.~Shah.
\newblock Randomized gossip algorithms.
\newblock {\em IEEE Transactions on Information Theory}, 52(6):2508--2530,
  2006.

\bibitem{RC-AC-LS-SZ:08b}
R.~Carli, A.~Chiuso, L.~Schenato, and S.~Zampieri.
\newblock A {PI} consensus controller for networked clocks synchronization.
\newblock In {\em {IFAC} {W}orld {C}ongress}, pages 10289--10294, Seoul, Korea,
  July 2008.

\bibitem{RC-FG-SZ:09}
R.~Carli, F.~Garin, and S.~Zampieri.
\newblock Quadratic indices for the analysis of consensus algorithms.
\newblock In {\em 4th Information Theory and Applications Workshop}, San Diego,
  CA, February 2009.

\bibitem{GC-MD-SG-NL-CN-TN:08}
G.~Chockler, M.~Demirbas, S.~Gilbert, N.~Lynch, C.~Newport, and T.~Nolte.
\newblock Consensus and collision detectors in radio networks.
\newblock {\em Distributed Computing}, 21(1):55--84, 2008.

\bibitem{PD-FB-PT-MV:08}
P.~Denantes, F.~Benezit, P.~Thiran, and M.~Vetterli.
\newblock Which distributed averaging algorithm should {I} choose for my sensor
  network?
\newblock In {\em IEEE Conference on Computer Communications (INFOCOM)}, pages
  986--994, April 2008.

\bibitem{PD-LSC:93}
P.~Diaconis and L.~Saloff-Coste.
\newblock Comparison theorems for reversible {M}arkov chains.
\newblock {\em The Annals of Applied Probability}, 3(3):696--730, 1993.

\bibitem{AGD-ADS-MJW:08}
A.~G. Dimakis, A.~D. Sarwate, and M.~J. Wainwright.
\newblock Geographic gossip: Efficient averaging for sensor networks.
\newblock {\em IEEE Transactions on Signal Processing}, 56(3):1205--1216, 2008.

\bibitem{OD-FB-PT:05}
O.~Dousse, F.~Baccelli, and P.~Thiran.
\newblock {Impact of interferences on connectivity in Ad Hoc Networks}.
\newblock {\em {IEEE-ACM} {T}ransactions on {N}etworking},
  {13}({2}):{425--436}, 2005.

\bibitem{OD-MF-NM-RM-PT:06}
O.~Dousse, M.~Franceschetti, N.~Macris, R.~Meester, and P.~Thiran.
\newblock {Percolation in the signal to interference ratio graph}.
\newblock {\em Journal of {A}pplied {P}robability}, {43}({2}):{552--562}, 2006.

\bibitem{FF-JCD:10}
F.~Fagnani and J.-C. Delvenne.
\newblock Democracy in {M}arkov chains and its preservation under local
  perturbations.
\newblock In {\em {IEEE} Conf. on Decision and Control}, 2010.
\newblock Submitted.

\bibitem{FF-SZ:07ita}
F.~Fagnani and S.~Zampieri.
\newblock Randomized consensus algorithms over large scale networks.
\newblock In {\em Information Theory and Applications Workshop}, San Diego, CA,
  2007.

\bibitem{FF-SZ:08a}
F.~Fagnani and S.~Zampieri.
\newblock Randomized consensus algorithms over large scale networks.
\newblock {\em IEEE Journal on Selected Areas in Communications},
  26(4):{634--649}, 2008.

\bibitem{FF-SZ:08}
F.~Fagnani and S.~Zampieri.
\newblock Average consensus with packet drop communication.
\newblock {\em SIAM Journal on Control and Optimization}, 48(1):102--133, 2009.

\bibitem{JAF:91}
J.~A. Fill.
\newblock Eigenvalue bounds on convergence to stationarity for nonreversible
  {M}arkov chains, with an application to the exclusion process.
\newblock {\em The Annals of Applied Probability}, 1(1):62--87, 1991.

\bibitem{MF-RM:07}
M.~Franceschetti and R.~Meester.
\newblock {\em Random networks for communication}.
\newblock {Cambridge University Press}, 2007.

\bibitem{FG-SZ:09}
F.~Garin and S.~Zampieri.
\newblock Performance of consensus algorithms in large-scale distributed
  estimation.
\newblock In {\em {E}uropean {C}ontrol {C}onference}, Budapest, Hungary, August
  2009.

\bibitem{YWH-AS:06}
Y.~W. Hong and A.~Scaglione.
\newblock Energy-efficient broadcasting with cooperative transmission in
  wireless sensor networks.
\newblock {\em {IEEE Transactions on Wireless Communications}},
  5(10):2844--2855, 2006.

\bibitem{DK-AD-JG:03}
D.~Kempe, A.~Dobra, and J.~Gehrke.
\newblock Gossip-based computation of aggregate information.
\newblock In {\em IEEE Symposium on Foundations of Computer Science}, pages
  482--491, Washington, DC, October 2003.

\bibitem{SK-AS-RJT:07}
S.~Kirti, A.~Scaglione, and R.~J. Thomas.
\newblock A scalable wireless communication architecture for average consensus.
\newblock In {\em {IEEE} Conf. on Decision and Control}, pages 32--37, New
  Orleans, LA, December 2007.

\bibitem{QL-DR:06}
Q.~Li and D.~Rus.
\newblock Global clock syncronization in sensor networks.
\newblock {\em IEEE Transactions on Computers}, 55(2):214-- 226, 2006.

\bibitem{DM-DS:08}
D.~Mosk-Aoyama and D.~Shah.
\newblock Fast distributed algorithms for computing separable functions.
\newblock {\em IEEE Transactions on Information Theory}, 54(7):2997--3007,
  2008.

\bibitem{BN-AGD-MG:08}
B.~Nazer, A.~G. Dimakis, and M.~Gastpar.
\newblock Local interference can accelerate gossip algorithms.
\newblock In {\em Allerton Conf. on Communications, Control and Computing},
  Allerton, IL, USA, September 2008.

\bibitem{SR:07}
S.~Rai.
\newblock The spectrum of a random geometric graph is concentrated.
\newblock {\em {Journal of Theoretical Probability}}, 20(2):119--132, 2007.

\bibitem{BR-RDA:04}
B.~Recht and R.~{D'Andrea}.
\newblock Distributed control of systems over discrete groups.
\newblock {\em IEEE Transactions on Automatic Control}, 49(9):1446--1452, 2004.

\bibitem{VS-MA-OS:06}
V.~Saligrama, M.~Alanyali, and O.~Savas.
\newblock Distributed detection in sensor networks with packet losses and
  finite capacity links.
\newblock {\em IEEE Transactions on Signal Processing}, 54(11):4118--4132,
  2006.

\bibitem{AT:99}
A.~Terras.
\newblock {\em Fourier analysis on finite groups and applications}, volume~43
  of {\em London Mathematical Society Student Texts}.
\newblock Cambridge University Press, Cambridge Ma, 1999.

\end{thebibliography}

\end{document}